\documentclass[11pt]{amsart}

\usepackage{amscd}
\usepackage{amsfonts}
\usepackage{amsmath}
\usepackage{amssymb}
\usepackage{amsthm}
\usepackage{color}
\usepackage{enumerate}
\usepackage{epsfig}
\usepackage{graphics}
\usepackage{latexsym}
\usepackage{mathrsfs}
\usepackage{pxfonts}
\usepackage{subfig}
\usepackage[english]{babel}
\usepackage[left=3cm,right=3cm,top=3cm,bottom=3cm]{geometry}
\usepackage{lscape}

\usepackage{algorithm}
\usepackage{algorithmic}

\theoremstyle{definition}
\newtheorem{Definition}{Definition}[section]

\theoremstyle{plain}

\newtheorem{Lemma}[Definition]{Lemma}
\newtheorem{Theorem}[Definition]{Theorem}

\numberwithin{equation}{section}
\allowdisplaybreaks[4]

\makeatletter
\newcommand{\biggg}[1]{{\hbox{$\left#1\vbox to 20.5pt{}\right.\n@space$}}}
\newcommand{\Biggg}[1]{{\hbox{$\left#1\vbox to 23.5pt{}\right.\n@space$}}}
\newcommand{\bigggg}[1]{{\hbox{$\left#1\vbox to 26.5pt{}\right.\n@space$}}}
\newcommand{\Bigggg}[1]{{\hbox{$\left#1\vbox to 29.5pt{}\right.\n@space$}}}
\newcommand{\biggggg}[1]{{\hbox{$\left#1\vbox to 32.5pt{}\right.\n@space$}}}
\newcommand{\Biggggg}[1]{{\hbox{$\left#1\vbox to 35.5pt{}\right.\n@space$}}}
\newcommand{\bigggggg}[1]{{\hbox{$\left#1\vbox to 38.5pt{}\right.\n@space$}}}
\newcommand{\Bigggggg}[1]{{\hbox{$\left#1\vbox to 41.5pt{}\right.\n@space$}}}
\makeatother

\begin{document}
\title[Higher-order error estimates of the discrete-time Clark--Ocone formula]{Higher-order error estimates of the discrete-time \\ Clark--Ocone formula}
\author{Tsubasa Nishimura}
\thanks{T. Nishimura. Kojimachi-odori Building 12F, 2-4-1 Kojimachi, Chiyoda-ku, Tokyo 102-0083, Japan, E-mail: \texttt{bldy.roze@gmail.com}}
\author{Kenji Yasutomi}
\thanks{K. Yasutomi. Department of Mathematical Sciences, Ritsumeikan University, 1-1-1, Nojihigashi, Kusatsu, Shiga, 525-8577, Japan, E-mail: \texttt{yasutomi@se.ritsumei.ac.jp}}
\author{Tomooki Yuasa$^{*}$}
\thanks{T. Yuasa (Corresponding Author). Department of Mathematical Sciences, Ritsumeikan University, 1-1-1, Nojihigashi, Kusatsu, Shiga, 525-8577, Japan, E-mail: \texttt{to-yuasa@fc.ritsumei.ac.jp}}
\date{}

\begin{abstract}
\label{sec:abst}
In this article, we investigate the convergence rate of the discrete-time Clark--Ocone formula provided by Akahori--Amaba--Okuma \cite{AJATOK}.
In that paper, they mainly focus on the $L_{2}$-convergence rate of the first-order error estimate related to the tracking error of the delta hedge in mathematical finance.
Here, as two extensions, we estimate ``the higher order error" for Wiener functionals with an integrability index $2$ and ``an arbitrary differentiability index."
\end{abstract}
\maketitle

\section{Introduction}
\label{sec:1}
Let $W \equiv (W_{t})_{t \geq 0}$ be a one-dimensional Wiener process on a complete probability space $(\Omega,{\mathcal F},{\mathbb P})$ such that ${\mathcal F}$ is the $\sigma$-filed generated by $W$ and $({\mathcal F}_{t}^{W})_{t \geq 0}$ be the filtration generated by $W$.
Then, as is well known, the following formula is obtained as one of the consequences of the Martingale representation theorem.

\begin{Theorem}[A corollary of the Martingale representation theorem (cf. \cite{INWS,MNTS})]
\label{theo:1.1}
Let $T>0$ and $F \in L_{2}(\Omega,{\mathcal F}_{T}^{W},{\mathbb P})$.
Then, there exists a one-dimensional $({\mathcal F}_{t}^{W})_{t \in [0,T]}$-predictable process $(f_{t})_{t \in [0,T]}$ such that ${\mathbb E}[\int_{0}^{T}f_{t}^{2}{\rm d}t]<\infty$ and
$$
F={\mathbb E}[F]+\int_{0}^{T}f_{s}{\rm d}W_{t}.
$$
\end{Theorem}

The following formula explicitly gives the integrand $(f_{t})_{t \in [0,T]}$ using the Malliavin derivative.

\begin{Theorem}[Clark--Ocone formula (cf. \cite{DNGOBMFP,ND,NDNE})]
\label{theo:1.2}
Let $T>0$ and $F \in {\mathbb D}_{2,1}$.
Then, it holds that
$$
F={\mathbb E}[F]+\int_{0}^{T}{\mathbb E}\left[D_{t}F \,\bigl|\, {\mathcal F}_{t}^{W}\right]{\rm d}W_{t}.
$$
\end{Theorem}

Here, ${\mathbb D}_{p,s}$ is a Sobolev-type space with the norm $\|\cdot\|_{p,s}$ in the sense of the Malliavin calculus, which is indexed by an integrability index $p \in {\mathbb N}$ and a differentiability index $s \in {\mathbb R}$, and $(D_{t}F)_{t \geq 0}$ is the process of the Malliavin (Fr\'echet) derivative of $F$ using the identification between the Hilbert spaces $L_{2}(\Omega; L_{2}([0,\infty)))$ and $L_{2}(\Omega \times [0,\infty))$.

In the context of mathematical finance, the Clark--Ocone formula is often used to find hedging portfolio strategies (cf. \cite{DNGOBMFP,ND,NDNE}).
However, owing to practical considerations and to the transaction costs, one can only use discrete-time hedging portfolio strategies.
Thus, the discrete-time strategy may have a different terminal value than the required payoff since the number of assets during the discretized time is fixed.
This causes an error called tracking error (or hedging error) of the delta hedge between the continuous-time and the discrete-time strategies.
For instance, for $F \in {\mathbb D}_{2,1}$ (e.g., European options), the tracking error is given by
\begin{align*}
\text{tracking error}&=\int_{0}^{T}{\mathbb E}\left[D_{t}F \,\bigl|\, {\mathcal F}_{t}^{W}\right]{\rm d}W_{t}-\sum_{\ell=1}^{N}{\mathbb E}\left[\left.D_{t_{\ell}}F \,\right|\, {\mathcal F}_{t_{\ell-1}}^{W}\right] \Delta W_{\ell}^{(N)} \\
&=F-{\mathbb E}[F]-\sum_{\ell=1}^{N}{\mathbb E}\left[\left.D_{t_{\ell}}F \,\right|\, {\mathcal F}_{t_{\ell-1}}^{W}\right] \Delta W_{\ell}^{(N)}.
\end{align*}
Here, $T>0$ is a maturity time, $0=t_{0}<t_{1}<\cdots<t_{N}=T$ is a discrete-time, $N \in {\mathbb N}$ is the number of partitions (time steps) of the closed interval $[0,T]$ and $\Delta W_{\ell}^{(N)}:=W_{t_{\ell}}-W_{t_{\ell-1}}$, $\ell \in \{1,2,\ldots,N\}$.
The analysis of such errors has been actively studied in the last 20 years (e.g., \cite{BMWM,GS,GSTA,GEMA,GETE,HTMP,HM,MAST,TE,ZR}).

As a discrete form of the Clark--Ocone formula, Akahori-Amaba-Okuma \cite{AJATOK} provided the following formula.

\begin{Theorem}[\cite{AJATOK}, Theorem 2.1, Discrete-time Clark--Ocone formula on ${\mathbb D}_{2,s}$]
\label{theo:1.3}
Let $T>0$, $N \in {\mathbb N}$, $s \in {\mathbb R}$ and $F \in {\mathbb D}_{2,s}^{(N)}$.
Then, it holds that
\begin{align}
\label{eq:1.1}
F-{\mathbb E}\left[F\right]=\sum_{m=1}^{\infty}\sum_{\ell=1}^{N}\frac{(T/N)^{m/2}}{\sqrt[]{m!}}{\mathbb E}\left[\left.D_{\Delta W_{\ell}^{(N)}}^{m}F \,\right|\, \{\Delta W_{i}^{(N)}\}_{i=1}^{\ell-1}\right]H_{m}(\frac{\Delta W_{\ell}^{(N)}}{\sqrt[]{T/N}}),
\end{align}
where the infinite sum converges in ${\mathbb D}_{2,s}^{(N)}$.
\end{Theorem}

Here, ${\mathbb D}_{2,s}^{(N)}$ is a Sobolev-type space with the norm $\|\cdot\|_{2,s}$ which is based on $\{\Delta W_{i}^{(N)}\}_{i=1}^{N}$, $D_{h}F$ is the G\^ateaux derivative of $F \in {\mathcal P}$ along $h \in L_{2}([0,\infty))$ ($\equiv$ the first Wiener chaos), and $H_{m}$ is the $m$-th Hermite polynomial.
The space ${\mathbb D}_{2,s}^{(N)}$ is defined in detail in Section \ref{sec:2}.

For $n \in {\mathbb N}$, we define the $n$-th-order error of the discrete-time Clark--Ocone formula by
$$
{\rm Err}_{n}^{(N)}(F):=F-\left({\mathbb E}[F]+\sum_{m=1}^{n}\sum_{\ell=1}^{N}\frac{(T/N)^{m/2}}{\sqrt[]{m!}}{\mathbb E}\left[\left.D_{\Delta W_{\ell}^{(N)}}^{m}F \,\right|\, \{\Delta W_{i}^{(N)}\}_{i=1}^{\ell-1}\right]H_{m}(\frac{\Delta W_{\ell}^{(N)}}{\sqrt[]{T/N}})\right).
$$
Here, we see that the first order error
$$
{\rm Err}_{1}^{(N)}(F) =F-\left({\mathbb E}[F] + \sum_{\ell=1}^{N} {\mathbb E}\left[\left.D_{\Delta W_{\ell}^{(N)}}F \,\right|\, \{\Delta W_{i}^{(N)}\}_{i=1}^{\ell-1}\right] \Delta W_{\ell}^{(N)} \right)
$$
is related to the tracking error.
These errors occur in approximating the stochastic integral in Theorem \ref{theo:1.2}.
Let $F \in {\mathbb D}_{2,1}$ and $F^{(N)}:={\mathbb E}[F \,|\, \{\Delta W_{i}^{(N)}\}_{i=1}^{N}]$.
Then, $F^{(N)} \in {\mathbb D}_{2,1}^{(N)}$, and we can approximate the stochastic integral
$$
\int_{0}^{T}{\mathbb E}\left[D_{t}F \,\bigl|\, {\mathcal F}_{t}^{W}\right]{\rm d}W_{t}
$$
by the $n$-th-order discretization
\begin{align*}
\sum_{m=1}^{n}\sum_{\ell=1}^{N}\frac{(T/N)^{m/2}}{\sqrt[]{m!}}{\mathbb E}\left[\left.D_{\Delta W_{\ell}^{(N)}}^{m}F^{(N)} \,\right|\, \{\Delta W_{i}^{(N)}\}_{i=1}^{\ell-1}\right]H_{m}(\frac{\Delta W_{\ell}^{(N)}}{\sqrt[]{T/N}}).
\end{align*}
Then, by using the discrete-time Clark--Ocone formula, the error is $F-F^{(N)}+{\rm Err}_{n}^{(N)}(F^{(N)})$.

Amaba \cite{AT} and Amaba--Liu--Makhlouf--Saidaoui \cite{ATLNNMAST} established discrete-time Clark--Ocone formulas for Poisson functionals and pure-jump L\'evy processes, respectively.
They investigated the $L_{2}$-convergence rate of the first-order error related to the tracking error of the delta hedge in mathematics finance.
In \cite{AJATOK}, they investigated it under the following stationary condition since it is difficult to find its rate for general Wiener functionals.

\begin{Definition}[\cite{AJATOK}, Section 3.3, stationary]
\label{def:1.4}
Let $T>0$, $\{F^{(N)}\}_{N \in {\mathbb N}}$ be a sequence of Wiener functionals (i.e., $F^{(N)} \in \cup_{s>0}{\mathbb D}_{2,-s}^{(N)}=:{\mathbb D}_{2,-\infty}^{(N)}$ for all $N \in {\mathbb N}$).
If $\{F^{(N)}\}_{N \in {\mathbb N}}$ satisfies the following condition, then it is called stationary: There exists $C>0$ such that for any $N \in {\mathbb N}$ and $m \in \{2,3,\ldots\}$,
$$
\sup_{\substack{a \in {\mathbb Z}_{+}^{N} \\ |a|=m}}{\mathbb E}\left[D_{\Delta W_{1}^{(N)}}^{a_{1}}D_{\Delta W_{2}^{(N)}}^{a_{2}} \cdots D_{\Delta W_{N}^{(N)}}^{a_{N}}F\right]^{2} \leq C\frac{m!\|J_{m}F^{(N)}\|_{2}^{2}}{T^{m}}.
$$
\end{Definition}
Here, $J_{m}$ is the projection from $L^{2}(\Omega,{\mathcal F},{\mathbb P})$ onto the $m$-th Wiener chaos, and $\|\cdot\|_{2}$ is the standard norm of $L^{2}(\Omega,{\mathcal F},{\mathbb P})$.
For example, a sequence composed of a one-dimensional functional $F(W_{T})$ is stationary (cf. \cite{AJATOK}, Section 3.3), but $F^{(N)}=\sum_{i=1}^{N}{\bf 1}_{[0,\infty)}(W_{iT/N})T/N$, $N \in {\mathbb N}$ is not stationary (cf. \cite{AJATOK}, Section 3.6).
As results of the $L_{2}$-convergence rate of the first-order error, \cite{AJATOK} provides the following estimates.

\begin{Theorem}[\cite{AJATOK}, Theorem 3.1, Theorem 3.3, Theorem 3.4, Theorem 3.8]
\label{theo:1.5}
\begin{enumerate}
\item Let $T>0$, $N \in {\mathbb N}$, $n \in {\mathbb N}$ and $F \in {\mathbb D}_{2,2+n}^{(N)}$.
Then, it holds that
$$
\left\|{\rm Err}_{n}^{(N)}(F)\right\|_{2} \leq \frac{(T\zeta(n+1)\int_{0}^{T}\|D_{t}^{n+1}F\|_{2}^{2}{\rm d}t)^{1/2}}{N^{1/2}},
$$
where $\zeta$ is the Riemann zeta function.
\item Let $T>0$, $r \in [0,1]$ and $F^{(N)} \in {\mathbb D}_{2,r}^{(N)}$ for all $N \in {\mathbb N}$.
If $\{F^{(N)}\}_{\mathbb N}$ is stationary, then there exists $C>0$ such that for any $N \in {\mathbb N}$,
$$
\left\|{\rm Err}_{1}^{(N)}(F^{(N)})\right\|_{2} \leq \frac{C\|F^{(N)}\|_{2,r}}{N^{r/2}}.
$$
\item Let $T>0$, $N_{0},N_{1} \in {\mathbb N}$, $r \in [0,1]$ and $F:=F(\Delta W_{1}^{(N_{0})},\Delta W_{2}^{(N_{0})},\ldots,\Delta W_{N_{0}}^{(N_{0})}) \in {\mathbb D}_{2,r}^{(N_{0})}$.
Then, it holds that
$$
\left\|{\rm Err}_{1}^{(N_{0}N_{1})}(F)\right\|_{2} \leq \frac{\|F\|_{2,r}}{N_{1}^{r/2}}.
$$
\item Let $T>0$ and $F^{(N)}:=\sum_{i=1}^{N}{\bf 1}_{[0,\infty)}(W_{iT/N})T/N$ for all $N \in {\mathbb N}$.
Then, there exists $C>0$ such that for any $N \in {\mathbb N}$,
$$
\left\|{\rm Err}_{1}^{(N)}(F^{(N)})\right\|_{2} \leq \frac{C}{N^{1/2}}.
$$
Note that in this case, the $L_{2}$-convergence rate of the first-order error is $N^{-1/2}$ even though $\{F^{(N)}\}_{N \in {\mathbb N}}$ is not stationary.
\end{enumerate}
\end{Theorem}
From these results, we can deduce that for $r \in [0,1]$ and $F \in {\mathbb D}_{2,r}^{(N)}$, the $L_{2}$-convergence rate of the first-order error is roughly $N^{-r/2}$.
Our arguments are based on Theorem \ref{theo:1.5} (3).
As two extensions, we estimate ``the higher order error" for Wiener functionals with an integrability index $2$ and ``an arbitrary differentiability index."
That is, strictly speaking, we will show that for $n \in {\mathbb N}$, $r \in [0,1]$ and $F \in {\mathbb D}_{2,s+rn}^{(N_{0})}$, the rate ${\mathbb D}_{2,s}$-convergence rate of ${\rm Err}_{n}^{(N_{0}N_{1})}(F^{(N_{0})})$ is $(n!N_{1}^{n})^{-r/2}$.
In particular, it is applicable to digital options as well as European options since it does not require the smoothness of Wiener functionals.

The discrete-time Clark--Ocone formula can also be applied to control variates method that is a variance reduction technique used in Monte Carlo methods.
Recently, Belomestny--H\"afner--Nagapetyan--Urusov \cite{BDHSNTUM} developed control variate schemes for discrete diffusion processes using analogues of the discrete-time Clark--Ocone formula.
The right-hand side ${\rm Err}_{\infty}^{(N)}(F):=\lim_{n \uparrow \infty}{\rm Err}_{n}^{(N)}(F)$ of \eqref{eq:1.1} is a perfect control variate for $F$ since the variance of $F-{\rm Err}_{\infty}^{(N)}(F)$ is zero.
However, in practice, one need to truncate the sum of \eqref{eq:1.1} in order to use this control variate.
In \cite{BDHSNTUM}, they avoided this problem by using weak approximation schemes, where the Brownian increments are replaced by simple discrete-valued random variables.
For instance, weak approximation schemes whose increments are replaced by Bernoulli random variables are described in \cite{KEPPE} and replaced by orthogonal random variables in \cite{AJKMSTYT}.
In addition, they computed conditional expectations in \eqref{eq:1.1} by regression methods and provided weak approximation schemes of first and second order.
The higher-order error estimate might be able to be applied to computational complexity arguments for such control variance schemes in order to better truncate the sum of \eqref{eq:1.1} with considering the computational cost.
Such control variate schemes are also seen in deep learning algorithms for solving parametric partial differential equations (cf. \cite{VMSSDSL}).

\subsection*{Outline}
\label{sec:1.1}
This article is divided as follows:
In Section \ref{sec:2}, we give some notations and definitions used throughout this article.
In Section \ref{sec:3}, we state the ${\mathbb D}_{2,s}$-convergence rate of the higher-order error of the discrete-time Clark--Ocone formula.
In Section \ref{sec:4}, we prove the statement mentioned in Section \ref{sec:3}.

\section{Notations and definitions}
\label{sec:2}
We give some notations and definitions used throughout this article.
The Kronecker delta is denoted by $\delta_{i,j}$, $i,j \in {\mathbb N}$.
We define by ${\mathfrak S}_{m}$ the symmetric group of degree $m \in {\mathbb N}$ on the set $\{1,2,\ldots,m\}$ (i.e., it denotes the set of all bijective mappings from $\{1,2,\ldots,m\}$ to $\{1,2,\ldots,m\}$).

We next introduce definitions and notations used in the Malliavin calculus.
For more detail, we refer to \cite{DNGOBMFP,INWS,MNTS,ND,NDNE,SI}.

\subsection{Malliavin calculus}
\label{sec:2.1}
Let $W \equiv (W_{t})_{t \geq 0}$ be a one-dimensional Wiener process on a complete probability space $(\Omega,{\mathcal F},{\mathbb P})$ such that ${\mathcal F}$ is the $\sigma$-filed generated by $W$ and $({\mathcal F}_{t}^{W})_{t \geq 0}$ be the filtration generated by $W$ (i.e., ${\mathcal F}_{t}^{W}=\sigma(W_{s} \,;\, 0 \leq s \leq t) \vee {\mathcal N}$, $t \geq 0$, where ${\mathcal N}$ is the set of all ${\mathbb P}$-null sets).
The space of all ${\mathbb R}$-valued square-integrable functions on $\Omega$ is denoted by $L_{2}(\Omega,{\mathcal F},{\mathbb P})$.
We define an inner product on $L_{2}(\Omega,{\mathcal F},{\mathbb P})$ by $\langle F,G \rangle_{2}:={\mathbb E}[FG]$, $F,G \in L_{2}(\Omega,{\mathcal F},{\mathbb P})$.
The norm on $L_{2}(\Omega,{\mathcal F},{\mathbb P})$ induced from its inner product is denoted by $\|\cdot\|_{2}$.
We define the space of all ${\mathbb R}$-valued square-integrable functions with respect to the Lebesgue measure on $[0,\infty)$ by
$$
{\mathbb H}:=L_{2}([0,\infty)).
$$
Using the martingale convergence theorem, we define an isonormal Gaussian process by
$$
W(h):=\int_{0}^{\infty}h(t){\rm d}W_{t}, \quad h \in {\mathbb H}.
$$
Note that the mapping $h \mapsto W(h)$ is a linear isometry from ${\mathbb H}$ onto the first Wiener chaos whose elements are zero-mean Gaussian random variables.
The space of all polynomial functionals is denoted by
$$
{\mathcal P}:=\left\{p(W(h_{1}),W(h_{1}),\ldots,W(h_{m})) \,;\, 
\begin{array}{l}
p: {\mathbb R}^{m} \to {\mathbb R} \text{ is real polynomial}, \\
\hspace{0.1cm}h_{1},h_{2},\ldots,h_{m} \in {\mathbb H} \text{ and } m \in {\mathbb N}
\end{array}
\right\}.
$$

Let ${\mathbb Z}_{+} \equiv {\mathbb N} \cup \{0\}$.
We define the set of multi-indexes by $\Lambda:=\cup_{\ell=1}^{\infty}{\mathbb Z}_{+}^{\ell}$.
For $\ell \in {\mathbb N}$ and a multi-index $a=(a_{1},a_{2},\ldots,a_{\ell}) \in {\mathbb Z}_{+}^{\ell}$, we formally identify it with $(a_{1},a_{2},\ldots,a_{\ell},0,0,\ldots)$.
For $a \in \Lambda$, we define its length and product by $|a|:=\sum_{i=1}^{\infty}a_{i}$ and $a!:=\prod_{\substack{i \in {\mathbb N} \\ a_{i} \neq 0}}a_{i}$, respectively.
The Hermite polynomials are denoted by
$$
H_{m}(x)=\frac{(-1)^{m}}{\sqrt[]{m!}}e^{\frac{x^{2}}{2}}\frac{{\rm d}^{m}}{{\rm d}x^{m}}e^{-\frac{x^{2}}{2}}, \quad x \in {\mathbb R}, \quad m \in {\mathbb N}
$$
and $H_{0} \equiv 1$.
The Fourier--Hermite polynomial with a multi-index $a \in \Lambda$ and an orthonormal system ${\bf h}=\{h_{i}\}_{i \in {\mathbb N}}$ of ${\mathbb H}$ is denoted by
$$
{\mathbf H}_{a}({\bf h}):=\prod_{i=1}^{\infty}H_{a_{i}}(W(h_{i})).
$$

Let $s \in {\mathbb R}$.
We define an inner product on ${\mathcal P}$ by
$$
\langle F,G \rangle_{2,s}:={\mathbb E}\left[(I-L)^{s/2}F(I-L)^{s/2}G\right], \quad F,G \in {\mathcal P},
$$
where $L: {\mathcal P} \to {\mathcal P}$ is the Ornstein--Uhlenbeck operator.
The norm on ${\mathcal P}$ induced from its inner product is denoted by $\|\cdot\|_{2,s}$.
The completion of ${\mathcal P}$ with respect to the norm $\|\cdot\|_{2,s}$ is denoted by ${\mathbb D}_{2,s} \equiv {\mathbb D}_{2,s}({\mathbb R})$.
In particular, if $s=0$, ${\mathbb D}_{2,0}$ coincides with $L_{2}(\Omega,{\mathcal F},{\mathbb P})$.
For a complete orthonormal system ${\bf h}$ of ${\mathbb H}$, the space ${\mathbb D}_{2,s}$ is a real separable Hilbert space with the complete orthonormal system $\{(1+|a|)^{-s/2}{\bf H}_{a}({\bf h}) \,;\,a \in \Lambda\}$.

Moreover, we set ${\mathbb D}_{2,\infty}:=\cap_{s>0}{\mathbb D}_{2,s}$ and ${\mathbb D}_{2,-\infty}:=\cup_{s>0}{\mathbb D}_{2,-s}$.
Then, the generalized expectation of $F \in {\mathbb D}_{2,-\infty}$ and $G \in {\mathbb D}_{2,\infty}$ is denoted by ${\mathbb E}[FG]:={}_{{\mathbb D}_{2,-\infty}}\langle F,G \rangle_{{\mathbb D}_{2,\infty}}$, noting that ${\mathbb D}_{2,-\infty}$ is the dual space of ${\mathbb D}_{2,\infty}$.
Note that if $F, G \in {\mathcal P}$, then the generalized expectation ${\mathbb E}[FG]$ coincides with the standard expectation (i.e., it is the value of the integral of $FG$ with respect to the probability measure ${\mathbb P}$).
In particular, the generalized expectation of $F \in {\mathbb D}_{2,-\infty}$ is denoted by ${\mathbb E}[F]:={}_{{\mathbb D}_{2,-\infty}}\langle F,1 \rangle_{{\mathbb D}_{2,\infty}}$.
For example, for any $s \in {\mathbb R}$, $a \in \Lambda$, an orthonormal system ${\bf h}$ of ${\mathbb H}$ and $F \in {\mathbb D}_{2,s}$, since ${\mathcal P}$ is a dense subspace of ${\mathbb D}_{2,s}$ and $(1-L)^{-s}{\bf H}_{a}({\bf h})=(1+|a|)^{-s}{\bf H}_{a}({\bf h})$, it holds that
\begin{align}
\label{eq:2.1}
{\mathbb E}[F{\bf H}_{a}({\bf h})]=(1+|a|)^{-s/2}\langle F,(1+|a|)^{-s/2}{\bf H}_{a}({\bf h}) \rangle_{2,s}.
\end{align}

The Malliavin (Fr\'echet) derivative of $F=p(W(h_{1},W(h_{2})),\ldots,W(h_{n})) \in {\mathcal P}$ is denoted by
$$
DF=\sum_{i=1}^{n}\partial_{i}p(W(h_{1}),W(h_{2}),\ldots,W(h_{n}))h_{i}.
$$
Using the identification between the Hilbert spaces $L_{2}(\Omega; {\mathbb H})$ and $L_{2}(\Omega \times [0,\infty))$, the process of the Malliavin derivative of $F$  is denoted by $(D_{t}F)_{t \geq 0}$.
The G\^ateaux derivative of $F \in {\mathcal P}$ along $h \in {\mathbb H}$ is denoted by $D_{h}F:=\langle DF,h \rangle_{{\mathbb H}}$.
For example, for any $N \in {\mathbb N}$, a multi-index $a \in {\mathbb Z}_{+}^{N}$, an orthonormal system ${\bf h}=\{h_{i}\}_{i \in {\mathbb N}}$ of ${\mathbb H}$ and $h \in {\mathbb H}$, it holds that
\begin{align}
\label{eq:2.2}
D_{h}{\bf H}_{a}({\bf h})=\sum_{i=1}^{N}\langle h,h_{i} \rangle_{{\mathbb H}}\sqrt[]{a_{i}}{\bf H}_{a-\delta^{(i)}}({\bf h}),
\end{align}
where $\delta^{(i)}=(\delta_{i,1},\delta_{i,2},\ldots,\delta_{i,N})$, $i \in \{1,2,\ldots,N\}$.

Next, we introduce the base spaces of the discrete-time Clark--Ocone formula.

\subsection{Discrete-time subspaces}
\label{sec:2.2}
Let $T>0$, $N \in {\mathbb N}$ be the number of partitions (time steps) of the closed interval $[0,T]$.
Define
$$
\Delta W_{i}^{(N)}:=W_{\frac{i}{N}T}-W_{\frac{i-1}{N}T}=W({\bf 1}_{(\frac{i-1}{N}T,\frac{i}{N}T]}) \quad \text{ and } \quad h_{i}^{(N)}:=\frac{{\bf 1}_{(\frac{i-1}{N}T,\frac{i}{N}T]}}{\sqrt[]{T/N}}, \quad i \in {\mathbb N}.
$$
Then, ${\bf h}^{(N)}:=\{h_{i}^{(N)}\}_{i \in {\mathbb N}}$ is an orthonormal system of ${\mathbb H}$.
Indeed, we obtain for any $i,j \in {\mathbb N}$,
\begin{align}
\label{eq:2.3}
\langle h_{i}^{(N)}, h_{j}^{(N)} \rangle_{{\mathbb H}}=\frac{\langle {\bf 1}_{(\frac{i-1}{N}T,\frac{i}{N}T]}, {\bf 1}_{(\frac{j-1}{N}T,\frac{j}{N}T]} \rangle_{{\mathbb H}}}{T/N}
=\delta_{i,j}.
\end{align}
Using the linear isometry $h \mapsto W(h)$, we often identify ${\bf 1}_{(\frac{i-1}{N}T,\frac{i}{N}T]}$ with $\Delta W_{i}^{(N)}$.
We define the $N$-th discrete-time subspaces of ${\mathbb H}$ by
$$
{\mathbb H}^{(N,\ell)}:={\rm span}\left\{h_{i}^{(N)} \,;\, i \in \{1,2,\ldots,\ell\}\right\}, \quad \ell \in {\mathbb N}.
$$

In the same way as Section \ref{sec:2.1}, we introduce definitions and notations in the Malliavin calculus for the discrete-time subspaces.
Let $\ell \in {\mathbb N}$.
The space of all polynomial functionals on ${\mathbb H}^{(N,\ell)}$ is denoted by ${\mathcal P}^{(N,\ell)}:={\rm span}\{{\bf H}_{a}({\bf h}^{(N)}) \,;\, a \in {\mathbb Z}_{+}^{\ell}\}$.
Let $s \in {\mathbb R}$.
The completion of ${\mathcal P}^{(N,\ell)}$ with respect to the norm $\|\cdot\|_{2,s}$ is denoted by ${\mathbb D}_{2,s}^{(N,\ell)}$.
Then, it is a real separable Hilbert space with the complete orthonormal system $\{(1+|a|)^{-s/2}{\bf H}_{a}({\bf h}^{(N)}) \,;\,a \in {\mathbb Z}_{+}^{\ell}\}$.
In the case of $\ell=N$, we use a simplified notation ${\mathbb D}_{2,s}^{(N)}:={\mathbb D}_{2,s}^{(N,N)}$ to avoid complexity.

The generalized conditional expectation ${\mathbb E}[F \,|\, \{\Delta W_{i}^{(N)}\}_{i=1}^{\ell}]$ for $F \in {\mathbb D}_{2,s}$ is denoted by the projection onto ${\mathbb D}_{2,s}^{(N)}$.
That is,
\begin{align}
\label{eq:2.4}
{\mathbb E}\left[F \,\bigl|\, \{\Delta W_{i}^{(N)}\}_{i=1}^{\ell}\right]:=\sum_{a \in {\mathbb Z}_{+}^{\ell}}{\mathbb E}\left[F{\bf H}_{a}({\bf h}^{(N)})\right]{\bf H}_{a}({\bf h}^{(N)}).
\end{align}

\section{Main statement}
\label{sec:3}
In this section, we state the discrete-time Clark--Ocone formula on ${\mathbb D}_{2,s}$ and its higher-order error estimate.

Akahori--Amaba--Okuma \cite{AJATOK} provided the following Clark--Ocone formula.
We prove it again since the proof on the space ${\mathbb D}_{2,s}$ is not detailed in that paper.

\begin{Theorem}[\cite{AJATOK}, Discrete-time Clark--Ocone formula on ${\mathbb D}_{2,s}$]
\label{theo:3.1}
Let $T>0$, $N \in {\mathbb N}$, $s \in {\mathbb R}$ and $F \in {\mathbb D}_{2,s}^{(N)}$.
Then, it holds that
$$
F-{\mathbb E}\left[F\right]=\sum_{m=1}^{\infty}\sum_{\ell=1}^{N}\frac{(T/N)^{m/2}}{\sqrt[]{m!}}{\mathbb E}\left[\left.D_{\Delta W_{\ell}^{(N)}}^{m}F \,\right|\, \{\Delta W_{i}^{(N)}\}_{i=1}^{\ell-1}\right]H_{m}(\frac{\Delta W_{\ell}^{(N)}}{\sqrt[]{T/N}}),
$$
where the infinite sum converges in ${\mathbb D}_{2,s}^{(N)}$.
\end{Theorem}
\begin{proof}
Since $\{(1+|a|)^{-s/2}{\bf H}_{a}({\bf h}^{(N)}) \,;\, a \in {\mathbb Z}_{+}^{N}\}$ is a complete orthonormal system of ${\mathbb D}_{2,s}^{(N)}$, we obtain the following Fourier expansion with the generalized expectation \eqref{eq:2.1}:
\begin{align}
\label{eq:3.1}
F=\sum_{a \in {\mathbb Z}_{+}^{N}}{\mathbb E}\left[F{\bf H}_{a}({\bf h}^{(N)})\right]{\bf H}_{a}({\bf h}^{(N)}).
\end{align}
Thus for $m \in {\mathbb N}$ and $\ell \in \{1,2,\ldots,N\}$, we have
\begin{align}
\label{eq:3.2}
D_{h_{\ell}^{(N)}}^{m}F= \sum_{a \in {\mathbb Z}_{+}^{N}}{\mathbb E}\left[F{\bf H}_{a}({\bf h}^{(N)})\right]D_{h_{\ell}^{(N)}}^{m} {\bf H}_{a}({\bf h}^{(N)}),
\end{align}
where the infinite sum converges in ${\mathbb D}_{2,s-m}^{(N)}$.
Here, for $a \in {\mathbb Z}_{+}^{N}$, by the derivative \eqref{eq:2.2} for the Fourier--Hermite polynomial, we obtain
$$
D_{h_{\ell}^{(N)}}^{m}{\bf H}_{a}({\bf h}^{(N)})=
\left\{
\begin{array}{ll}
\displaystyle{\sqrt[]{\frac{a_{\ell}!}{(a_{\ell}-m)!}}{\bf H}_{a-m\delta^{(\ell)}}({\bf h}^{(N)})} & \text{if } m \leq a_{\ell} \\
0 & \text{if } m>a_{\ell},
\end{array}
\right.
$$
where $\delta^{(\ell)}=(\delta_{\ell,1},\delta_{\ell,2},\ldots,\delta_{\ell,N})$.
Hence, since $\{(1+|a|)^{-s/2}{\bf H}_{a}({\bf h}^{(N)}) \,;\, a \in {\mathbb Z}_{+}^{N}\}$ is a complete orthonormal system of ${\mathbb D}_{2,s}^{(N)}$, by using the general expectation \eqref{eq:2.1}, we have  for any $a' \in {\mathbb Z}_{+}^{N}$,
$$
{\mathbb E}\left[D_{h_{\ell}^{(N)}}^{m}{\bf H}_{a}({\bf h}^{(N)}){\bf H}_{a'}({\bf h}^{(N)})\right] \\
=
\left\{
\begin{array}{ll}
\sqrt[]{m!} & \text{if } (a_{1},a_{2},\ldots,a_{\ell})=(a',m) \text{ and } a_{\ell+1}=a_{\ell+2}=\cdots=a_{N}=0 \\
0 & \text{otherwise}.
\end{array}
\right.
$$
Therefore, by \eqref{eq:3.2}, we obtain the following integration by parts formula:
\begin{align}
\label{eq:3.3}
{\mathbb E}\left[D_{{\bf 1}_{(\frac{\ell-1}{N}T,\frac{\ell}{N}T]}}^{m}F{\bf H}_{a}({\bf h}^{(N)})\right]
&=\frac{1}{(T/N)^{m/2}}{\mathbb E}\left[D_{h_{\ell}^{(N)}}^{m}F{\bf H}_{a}({\bf h}^{(N)})\right] \\ \notag
&=\frac{\sqrt[]{m!}}{(T/N)^{m/2}}{\mathbb E}\left[F{\bf H}_{(a,m)}({\bf h}^{(N)})\right],
\end{align}
where $(a,m)=(a_{1},a_{2},\ldots,a_{\ell-1},m)$.

On the other hand, we order the infinite sum \eqref{eq:3.1} of the Fourier expansion as follows:
\begin{align}
\label{eq:3.4}
F={\mathbb E}[F]+\sum_{\ell=1}^{N}\sum_{m=1}^{\infty}\sum_{a \in {\mathbb Z}_{+}^{\ell-1}}{\mathbb E}\left[F{\bf H}_{(a,m)}({\bf h}^{(N)})\right]{\bf H}_{(a,m)}({\bf h}^{(N)}).
\end{align}
Here, by using the integration by parts formula \eqref{eq:3.3} and the generalized conditional expectation \eqref{eq:2.4}, we have
\begin{align}
\label{eq:3.5}
&\sum_{a \in {\mathbb Z}_{+}^{\ell-1}}{\mathbb E}\left[F{\bf H}_{(a,m)}({\bf h}^{(N)})\right]{\bf H}_{(a,m)}({\bf h}^{(N)}) \\ \notag
&=\frac{(T/N)^{m/2}}{\sqrt[]{m!}}\sum_{a \in {\mathbb Z}_{+}^{\ell-1}}{\mathbb E}\left[D_{{\bf 1}_{(\frac{\ell-1}{N}T,\frac{\ell}{N}T]}}^{m}F{\bf H}_{a}({\bf h}^{(N)})\right]{\bf H}_{a}({\bf h}^{(N)})H_{m}(W(h_{\ell}^{(N)})) \\ \notag
&=\frac{(T/N)^{m/2}}{\sqrt[]{m!}}{\mathbb E}\left[\left.D_{{\bf 1}_{(\frac{\ell-1}{N}T,\frac{\ell}{N}T]}}^{m}F \,\right|\, \{\Delta W_{i}^{(N)}\}_{i=1}^{\ell-1}\right]H_{m}(\frac{\Delta W_{\ell}^{(N)}}{\sqrt[]{T/N}}).
\end{align}
Using the identification of ${\bf 1}_{(\frac{\ell-1}{N}T,\frac{\ell}{N}T]}$ and $\Delta W_{\ell}^{(N)}$, we conclude the statement from \eqref{eq:3.4} and \eqref{eq:3.5}.
\end{proof}

Recall the error functions of the discrete-time Clark--Ocone formula:
\begin{align*}
{\rm Err}_{n}^{(N)}(F)&=F-\left({\mathbb E}[F]+\sum_{m=1}^{n}\sum_{\ell=1}^{N}\frac{(T/N)^{m/2}}{\sqrt[]{m!}}{\mathbb E}\left[\left.D_{\Delta W_{\ell}^{(N)}}^{m}F \,\right|\, \{\Delta W_{i}^{(N)}\}_{i=1}^{\ell-1}\right]H_{m}(\frac{\Delta W_{\ell}^{(N)}}{\sqrt[]{T/N}})\right) \\ \notag
&=\sum_{m=n+1}^{\infty}\sum_{\ell=1}^{N}\sum_{a \in {\mathbb Z}_{+}^{\ell-1}}{\mathbb E}\left[F{\bf H}_{(a,m)}({\bf h}^{(N)})\right]{\bf H}_{(a,m)}({\bf h}^{(N)})
\end{align*}
for $N \in {\mathbb N}$ and $n \in {\mathbb N}$.
The following theorem is our main statement.

\begin{Theorem}
\label{theo:3.2}
Let $T>0$, $N_{0},N_{1} \in {\mathbb N}$, $n \in {\mathbb N}$, $s \in {\mathbb R}$, $r \in [0,1]$ and $F \in {\mathbb D}_{2,s+rn}^{(N_{0})}$.
Then, it holds that
$$
\left\|{\rm Err}_{n}^{(N_{0}N_{1})}(F)\right\|_{2,s} \leq \frac{\|F\|_{2,s+rn}}{\sqrt[]{n!N_{1}^{n}}^{r}}.
$$
\end{Theorem}

Theorem \ref{theo:3.2} is a generalization of Theorem 3.4 in \cite{AJATOK} (see Theorem \ref{theo:1.5} (3) in this paper) about the order of the error and the differentiability index.
We prove it in the next section.

\section{Proof of Theorem \ref{theo:3.2}}
\label{sec:4}
In preparation for the proof of Theorem \ref{theo:3.2}, we introduce the symmetric form and the Hilbert--Schmidt inner product and give some lemmas.

\begin{Definition}[symmetric form]
\label{def:4.1}
Let $\ell \in {\mathbb N}$, $a \in {\mathbb Z}_{+}^{\ell}$ be a multi-index and ${\bf h}=\{h_{i}\}_{i \in {\mathbb N}}$ be an orthonormal system of ${\mathbb H}$.
Then, we define the $|a|$-th symmetric form of ${\bf h}$ by
\begin{align}
\label{eq:4.1}
{\bf h}_{a}(e_{1},e_{2},\ldots,e_{|a|})
:=\frac{1}{\sqrt[]{a!|a|!}}\sum_{\sigma \in {\mathfrak S}_{|a|}}(h_{1}^{a_{1}} \otimes h_{2}^{a_{2}} \otimes \cdots \otimes h_{\ell}^{a_{\ell}})(e_{\sigma(1)},e_{\sigma(2)},\ldots,e_{\sigma(|a|)})
\end{align}
for $(e_{1},e_{2},\ldots,e_{|a|}) \in {\mathbb H}^{|a|}$, where 
\begin{align}
\label{eq:4.2}
(h_{1}^{a_{1}} \otimes h_{2}^{a_{2}} \otimes \cdots \otimes h_{\ell}^{a_{\ell}})(e_{1},e_{2},\ldots,e_{|a|})
:=\prod_{i=1}^{\ell}\prod_{j=1}^{a_{i}}\langle h_{i},e_{\sum_{k=1}^{i-1}a_{k}+j} \rangle_{{\mathbb H}}.
\end{align}
\end{Definition}

In order to simplify the forms of \eqref{eq:4.1} and \eqref{eq:4.2}, we introduce the following notation:

$$
\eta_{a}(n):=i \quad \text{if} \quad \sum_{k=1}^{i-1}a_{k}<n \leq \sum_{k=1}^{i}a_{k}
$$
for $n \in \{1,2,\ldots,|a|\}$.
Using the above notation, \eqref{eq:4.2} and \eqref{eq:4.1} can be represented as follows:
$$
(h_{1}^{a_{1}} \otimes h_{2}^{a_{2}} \otimes \cdots \otimes h_{\ell}^{a_{\ell}})(e_{1},e_{2},\ldots,e_{|a|})
=\prod_{n=1}^{|a|}\langle h_{\eta_{a}(n)},e_{n} \rangle_{{\mathbb H}}
$$
and
\begin{align}
\label{eq:4.3}
{\bf h}_{a}(e_{1},e_{2},\ldots,e_{|a|})
&=\frac{1}{\sqrt[]{a!|a|!}}\sum_{\sigma \in {\mathfrak S}_{|a|}}\prod_{n=1}^{|a|}\langle h_{\eta_{a}(n)},e_{\sigma(n)} \rangle_{{\mathbb H}} \\ \notag
&=\frac{1}{\sqrt[]{a!|a|!}}\sum_{\sigma \in {\mathfrak S}_{|a|}}\prod_{n=1}^{|a|}\langle h_{\eta_{a}(\sigma^{-1}(n))},e_{n} \rangle_{{\mathbb H}} \\ \notag
&=\frac{1}{\sqrt[]{a!|a|!}}\sum_{\sigma \in {\mathfrak S}_{|a|}}\prod_{n=1}^{|a|}\langle h_{\eta_{a}(\sigma(n))},e_{n} \rangle_{{\mathbb H}}.
\end{align}

\begin{Definition}[Hilbert--Schmidt inner product]
\label{def:4.2}
Let $\ell,\ell' \in {\mathbb N}$, $a \in {\mathbb Z}_{+}^{\ell}$ and $a' \in {\mathbb Z}_{+}^{\ell'}$ be multi-indexes such that $|a|=|a'|=:m$, and ${\bf h}=\{h_{i}\}_{i=1}^{\ell}$ and ${\bf h}'=\{h_{i}'\}_{i=1}^{\ell'}$ be orthonormal systems of ${\mathbb H}$.
Then, we define the Hilbert--Schmidt inner product of the $m$-th symmetric forms ${\bf h}_{a}$ and ${\bf h}'_{a'}$ by
$$
\langle {\bf h}_{a},{\bf h}'_{a'} \rangle_{{\rm HS}}:=\sum_{i_{1},i_{2},\ldots,i_{m} \in {\mathbb N}}{\bf h}_{a}(e_{i_{1}},e_{i_{2}},\ldots,e_{i_{m}}){\bf h}'_{a'}(e_{i_{1}},e_{i_{2}},\ldots,e_{i_{m}}),
$$
where $\{e_{i}\}_{i \in {\mathbb N}}$ is any complete orthonormal system of ${\mathbb H}$.
\end{Definition}

The Hilbert--Schmidt inner product is well defined.
Namely, it does not depend on the complete orthonormal system $\{e_{i}\}_{i \in {\mathbb N}}$.
Indeed, by \eqref{eq:4.3}, we obtain
\begin{align}
\label{eq:4.4}
\langle {\bf h}_{a},{\bf h}'_{a'} \rangle_{{\rm HS}}
&=\frac{1}{\sqrt[]{a!a'!}\,m!}\sum_{\sigma \in {\mathfrak S}_{m}}\sum_{\sigma' \in {\mathfrak S}_{m}}\prod_{n=1}^{m}\sum_{i=1}^{\infty}\langle h_{\eta_{a}(\sigma(n))},e_{i} \rangle_{{\mathbb H}}\langle h_{\eta_{a'}(\sigma'(n))}',e_{i} \rangle_{{\mathbb H}} \\ \notag
&=\frac{1}{\sqrt[]{a!a'!}\,m!}\sum_{\sigma \in {\mathfrak S}_{m}}\sum_{\sigma' \in {\mathfrak S}_{m}}\prod_{n=1}^{m}\langle h_{\eta_{a}(\sigma(n))},h_{\eta_{a'}(\sigma'(n))}' \rangle_{{\mathbb H}} \\ \notag
&=\frac{1}{\sqrt[]{a!a'!}\,m!}\sum_{\sigma \in {\mathfrak S}_{m}}\sum_{\sigma' \in {\mathfrak S}_{m}}\prod_{n=1}^{m}\langle h_{\eta_{a}(n)},h_{\eta_{a'}(\sigma'(\sigma^{-1}(n)))}' \rangle_{{\mathbb H}} \\ \notag
&=\frac{1}{\sqrt[]{a!a'!}}\sum_{\sigma \in {\mathfrak S}_{m}}\prod_{n=1}^{m}\langle h_{\eta_{a}(n)},h_{\eta_{a'}(\sigma(n))}' \rangle_{{\mathbb H}}.
\end{align}

The following lemma is a result about the relationship between the $L_{2}$-inner product and the Hilbert--Schmidt inner product and also implies an integration by parts formula.

\begin{Lemma}
\label{lem:4.3}
Let $\ell,\ell' \in {\mathbb N}$, $a \in {\mathbb Z}_{+}^{\ell}$ and $a' \in {\mathbb Z}_{+}^{\ell'}$ be multi-indexes, and ${\bf h}=\{h_{i}\}_{i \in {\mathbb N}}$ and ${\bf h}'=\{h_{i}'\}_{i \in {\mathbb N}}$ be orthonormal systems of ${\mathbb H}$.
Then, it holds that
$$
{\mathbb E}[{\bf H}_{a}({\bf h}){\bf H}_{a'}({\bf h}')]=
\left\{
\begin{array}{ll}
\langle {\bf h}_{a},{\bf h}'_{a'} \rangle_{{\rm HS}} & \text{if } |a|=|a'| \\
0 & \text{if } |a| \neq |a'|.
\end{array}
\right.
$$
\end{Lemma}
\begin{proof}
We first take an orthonormal system $\{h_{i}''\}_{i=1}^{\ell''}$ of ${\mathbb H}$ which can represent both ${\bf h}$ and ${\bf h}'$.
Indeed, from the Gram--Schmidt orthonormalization, there exist $\ell'' \in {\mathbb N}$ and an orthonormal system $\{h_{i}''\}_{i=1}^{\ell''}$ of ${\mathbb H}$ such that $\max\{\ell,\ell'\} \leq \ell'' \leq \ell+\ell'$,
$$
h_{i}=h_{i}'', \quad i \in \{1,2,\ldots,\ell\} \quad \text{ and } \quad h_{j}'=\sum_{k=1}^{\ell''}\langle h_{j}',h_{k}'' \rangle_{{\mathbb H}}h_{k}'', \quad j \in \{1,2,\ldots,\ell'\}.
$$
Then, since $\{W(h_{i}'')\}_{i=1}^{\ell''}$ are independent zero-mean Gaussian random variables, we obtain
\begin{align}
\label{eq:4.5}
{\mathbb E}[{\bf H}_{a}({\bf h}){\bf H}_{a'}({\bf h}')]=\int_{{\mathbb R}^{\ell''}}\left(\prod_{i=1}^{\ell}H_{a_{i}}(x_{i})\right)\left(\prod_{j=1}^{\ell'}H_{a_{j}'}(y_{j})\right)\left(\prod_{i=1}^{\ell''}g(x_{i})\right){\rm d}x_{1}{\rm d}x_{2} \cdots {\rm d}x_{\ell''},
\end{align}
where $g(x)=\frac{1}{\sqrt[]{2\pi}}e^{-\frac{x^{2}}{2}}$, $x \in {\mathbb R}$ and $y_{j} \equiv y_{j}(x):=\sum_{k=1}^{\ell''}\langle h_{j}',h_{k}'' \rangle_{{\mathbb H}}x_{k}$, $x \in {\mathbb R}^{\ell''}$ for $j \in \{1,2,\ldots,\ell'\}$.
Here, by using the integration by parts and the general Leibniz rule, we have
\begin{align*}
\int_{{\mathbb R}}H_{a_{1}}(x_{1})\left(\prod_{j=1}^{\ell'}H_{a_{j}'}(y_{j})\right)g(x_{1}){\rm d}x_{1}
&=\frac{(-1)^{a_{1}}}{\sqrt[]{a_{1}!}}\int_{{\mathbb R}}\left(\prod_{j=1}^{\ell'}H_{a_{j}'}(y_{j})\right)\frac{{\rm d}^{a_{1}}g(x_{1})}{{\rm d}x_{1}^{a_{1}}}{\rm d}x_{1} \\
&=\frac{1}{\sqrt[]{a_{1}!}}\int_{{\mathbb R}}\left(\frac{{\rm d}^{a_{1}}}{{\rm d}x_{1}^{a_{1}}}\left(\prod_{j=1}^{\ell'}H_{a_{j}'}(y_{j})\right)\right)g(x_{1}){\rm d}x_{1} \\
&\hspace{-3.4cm}=\frac{1}{\sqrt[]{a_{1}!}}\int_{{\mathbb R}}\left(\sum_{\substack{k_{1}^{(1)},k_{2}^{(1)},\ldots,k_{\ell'}^{(1)} \in {\mathbb Z}_{+} \\ k_{1}^{(1)}+k_{2}^{(1)}+\cdots+k_{\ell'}^{(1)}=a_{1}}}\frac{a_{1}!}{k_{1}^{(1)}!k_{2}^{(1)}! \cdots k_{\ell'}^{(1)}!}\left(\prod_{j=1}^{\ell'}\frac{{\rm d}^{k_{j}^{(1)}}H_{a_{j}'}(y_{j})}{{\rm d}x_{1}^{k_{j}^{(1)}}}\right)\right)g(x_{1}){\rm d}x_{1} \\
&\hspace{-3.4cm}=\sqrt[]{a_{1}!}\int_{{\mathbb R}}\left(\sum_{\substack{k_{1}^{(1)},k_{2}^{(1)},\ldots,k_{\ell'}^{(1)} \in {\mathbb Z}_{+} \\ k_{1}^{(1)}+k_{2}^{(1)}+\cdots+k_{\ell'}^{(1)}=a_{1}}}\left(\prod_{j=1}^{\ell'}\frac{\langle h_{j}',h_{1} \rangle_{{\mathbb H}}^{k_{j}^{(1)}}}{k_{j}^{(1)}!}\left.\frac{{\rm d}^{k_{j}^{(1)}}H_{a_{j}'}(x)}{{\rm d}x^{k_{j}^{(1)}}}\right|_{x=y_{j}}\right)\right)g(x_{1}){\rm d}x_{1} \\
&\hspace{-3.4cm}=\sqrt[]{a_{1}!}\sum_{\substack{k_{1}^{(1)},k_{2}^{(1)},\ldots,k_{\ell'}^{(1)} \in {\mathbb Z}_{+} \\ k_{1}^{(1)}+k_{2}^{(1)}+\cdots+k_{\ell'}^{(1)}=a_{1}}}\left(\prod_{j=1}^{\ell'}\frac{\langle h_{j}',h_{1} \rangle_{{\mathbb H}}^{k_{j}^{(1)}}}{k_{j}^{(1)}!}\right)\int_{{\mathbb R}}\left(\prod_{j=1}^{\ell'}\left.\frac{{\rm d}^{k_{j}^{(1)}}H_{a_{j}'}(x)}{{\rm d}x^{k_{j}^{(1)}}}\right|_{x=y_{j}}\right)g(x_{1}){\rm d}x_{1}.
\end{align*}
Consequently, by iterating this calculation in \eqref{eq:4.5}, we obtain
\begin{align*}
{\mathbb E}[{\bf H}_{a}({\bf h}){\bf H}_{a'}({\bf h}')]
&=\sqrt[]{a!}\sum_{\substack{(k_{j}^{(i)}) \in {\mathbb Z}_{+}^{\ell} \otimes {\mathbb Z}_{+}^{\ell'} \\ k_{1}^{(i)}+k_{2}^{(i)}+\cdots+k_{\ell'}^{(i)}=a_{i} \forall i \in \{1,2,\ldots,\ell\}}}\left(\prod_{j=1}^{\ell'}\prod_{i=1}^{\ell}\frac{\langle h_{j}',h_{i} \rangle_{{\mathbb H}}^{k_{j}^{(i)}}}{k_{j}^{(i)}!}\right) \\
&\hspace{0.4cm}\times \int_{{\mathbb R}^{\ell''}}\left(\prod_{j=1}^{\ell'}\left.\frac{{\rm d}^{k_{j}^{(1)}+k_{j}^{(2)}+\cdots+k_{j}^{(\ell)}}H_{a_{j}'}(x)}{{\rm d}x^{k_{j}^{(1)}+k_{j}^{(2)}+\cdots+k_{j}^{(\ell)}}}\right|_{x=y_{j}}\right)\left(\prod_{i=1}^{\ell''}g(x_{i})\right){\rm d}x_{1}{\rm d}x_{2} \cdots {\rm d}x_{\ell''},
\end{align*}
where ${\mathbb Z}_{+}^{\ell} \otimes {\mathbb Z}_{+}^{\ell'}$ is the set of $\ell \times \ell'$ matrices whose elements are natural numbers.
Thus by the independence of $\{W(h_{j}')\}_{j=1}^{\ell'}$ and using the properties of Hermite polynomials, we have
\begin{align}
\label{eq:4.6}
{\mathbb E}[{\bf H}_{a}({\bf h}){\bf H}_{a'}({\bf h}')]
&=\sqrt[]{a!}\sum_{\substack{(k_{j}^{(i)}) \in {\mathbb Z}_{+}^{\ell} \otimes {\mathbb Z}_{+}^{\ell'} \\ k_{1}^{(i)}+k_{2}^{(i)}+\cdots+k_{\ell'}^{(i)}=a_{i} \forall i \in \{1,2,\ldots,\ell\}}}\left(\prod_{j=1}^{\ell'}\prod_{i=1}^{\ell}\frac{\langle h_{j}',h_{i} \rangle_{{\mathbb H}}^{k_{j}^{(i)}}}{k_{j}^{(i)}!}\right) \\ \notag
&\hspace{0.4cm}\times \prod_{j=1}^{\ell'}\int_{{\mathbb R}}\frac{{\rm d}^{k_{j}^{(1)}+k_{j}^{(2)}+\cdots+k_{j}^{(\ell)}}H_{a_{j}'}(x_{j})}{{\rm d}x_{j}^{k_{j}^{(1)}+k_{j}^{(2)}+\cdots+k_{j}^{(\ell)}}}g(x_{j}){\rm d}x_{j} \\ \notag
&=\sqrt[]{a!}\sum_{\substack{(k_{j}^{(i)}) \in {\mathbb Z}_{+}^{\ell} \otimes {\mathbb Z}_{+}^{\ell'} \\ k_{1}^{(i)}+k_{2}^{(i)}+\cdots+k_{\ell'}^{(i)}=a_{i} \forall i \in \{1,2,\ldots,\ell\}}}\left(\prod_{j=1}^{\ell'}\prod_{i=1}^{\ell}\frac{\langle h_{j}',h_{i} \rangle_{{\mathbb H}}^{k_{j}^{(i)}}}{k_{j}^{(i)}!}\right)\prod_{j=1}^{\ell'}\sqrt[]{a_{j}'!}\,\delta_{a_{j}',k_{j}^{(1)}+k_{j}^{(2)}+\cdots+k_{j}^{(\ell)}} \\ \notag
&=\sqrt[]{a!a'!}\sum_{\substack{(k_{j}^{(i)}) \in {\mathbb Z}_{+}^{\ell} \otimes {\mathbb Z}_{+}^{\ell'} \\ k_{1}^{(i)}+k_{2}^{(i)}+\cdots+k_{\ell'}^{(i)}=a_{i} \forall i \in \{1,2,\ldots,\ell\} \\ k_{j}^{(1)}+k_{j}^{(2)}+\cdots+k_{j}^{(\ell)}=a_{j}' \forall j \in \{1,2,\ldots,\ell'\}}}\left(\prod_{j=1}^{\ell'}\prod_{i=1}^{\ell}\frac{\langle h_{j}',h_{i} \rangle_{{\mathbb H}}^{k_{j}^{(i)}}}{k_{j}^{(i)}!}\right).
\end{align}
In particular, if $|a| \neq |a'| $, then ${\mathbb E}[{\bf H}_{a}({\bf h}){\bf H}_{a'}({\bf h}')]=0$ since the set of the above sum is empty.

Next, we consider it in terms of the Hilbert--Schmidt inner product under the assumption $|a|=|a'|=:m$.
To reparameterize the sum of the Hilbert--Schmidt inner product \eqref{eq:4.4}, we introduce the mapping
$$
{\mathfrak S}_{m} \ni \sigma \mapsto (\eta_{a}(\sigma(1)),\eta_{a}(\sigma(2)),\ldots,\eta_{a}(\sigma(m))) \in \{1,2,\ldots,\ell\}^{m}.
$$
The image of the mapping is
$$
\left\{{\bf i} \in \{1,2,\ldots,\ell\}^{m} \,;\, \sharp_{i}{\bf i}=a_{i} \text{ for all } i \in \{1,2,\ldots,\ell\}\right\},
$$
where 
$$
\sharp_{i}{\bf i}:=\sharp\{n \in \{1,2,\ldots,m\} \,;\, i_{n}=i\}
$$
for ${\bf i}=(i_{1},i_{2},\ldots,i_{m}) \in \{1,2,\ldots,\ell\}^{m}$ and $i \in \{1,2,\ldots,\ell\}$.

Consider the situation shown in the following figure, where there are $m$ balls named with numbers from $1$ to $\ell$, and for each $i \in \{1,2,\ldots,\ell\}$, the number of balls with the number $i$ is $a_{i}$.
\begin{align}
\label{eq:4.7}
\underbrace{\textcircled{\scriptsize $1$}\textcircled{\scriptsize $1$} \cdots \textcircled{\scriptsize $1$}}_{a_{1}} \underbrace{\textcircled{\scriptsize $2$}\textcircled{\scriptsize $2$} \cdots \textcircled{\scriptsize $2$}}_{a_{2}} \cdots \underbrace{\textcircled{\scriptsize $\ell$}\textcircled{\scriptsize $\ell$} \cdots \textcircled{\scriptsize $\ell$}}_{a_{\ell}}
\end{align}
Then, the image of the mapping implies the set of permutations of the balls.
Thus, since the number of elements in the image is the total number of the permutations, we obtain
\begin{align}
\label{eq:4.8}
\sharp\left\{{\bf i} \in \{1,2,\ldots,\ell\}^{m} \,;\, \sharp_{i}{\bf i}=a_{i} \text{ for all } i \in \{1,2,\ldots,\ell\}\right\}=\frac{m!}{a_{1}!a_{2}! \cdots a_{\ell}!}=\frac{m!}{a!}.
\end{align}
On the other hand, the mapping $\eta_{a}: \{1,2,\ldots,m\} \to \{1,2,\ldots,\ell\}$ gives a permutation of the balls lined up neatly in order from $1$ to $\ell$.
Indeed, by the definition of $\eta_{a}$, the permutation of the elements of $\eta_{a}$ coincides with the permutation \eqref{eq:4.7} as follows:
\begin{align}
\label{eq:4.9}
(\eta_{a}(1),\eta_{a}(2),\ldots,\eta_{a}(m))=(\underbrace{1,1,\cdots,1}_{a_{1}},\underbrace{2,2,\cdots,2}_{a_{2}},\cdots,\underbrace{\ell,\ell,\cdots,\ell}_{a_{\ell}}).
\end{align}
Note that we cannot distinguish between the balls with the same number in the permutation \eqref{eq:4.7}, but we can it in the permutation \eqref{eq:4.9}.
Fix a point ${\bf i}=(i_{1},i_{2},\ldots,i_{m})$ of the image (a permutation of the balls) and $\sigma \in {\mathfrak S}_{m}$ such that ${\bf i}=(\eta_{a}(\sigma(1)),\eta_{a}(\sigma(2)),\ldots,\eta_{a}(\sigma(m)))$.
Then, the mapping $\sigma$ can be understood as changing the order of the nicely aligned permutation \eqref{eq:4.9} to the permutation $(i_{1},i_{2},\cdots,i_{m})$.
In this case, since we can distinguish between the balls with the same number, the number of elements in the inverse image from ${\bf i}$ is
\begin{align}
\label{eq:4.95}
\sharp\left\{\sigma \in {\mathfrak S}_{m} \,;\, {\bf i}=(\eta_{a}(\sigma(1)),\eta_{a}(\sigma(2)),\ldots,\eta_{a}(\sigma(m)))\right\}=\prod_{i=1}^{\ell}a_{i}!=a!
\end{align}

The sum of the Hilbert--Schmidt inner product \eqref{eq:4.4} is reparameterized as follows: By \eqref{eq:4.95}, we obtain
\begin{align}
\label{eq:4.10}
\langle {\bf h}_{a},{\bf h}'_{a'} \rangle_{{\rm HS}}
&=\frac{1}{\sqrt[]{a!a'!}\,m!}\sum_{\sigma \in {\mathfrak S}_{m}}\sum_{\sigma' \in {\mathfrak S}_{m}}\prod_{n=1}^{m}\langle h_{\eta_{a}(\sigma(n))},h_{\eta_{a'}(\sigma'(n))}' \rangle_{{\mathbb H}} \\ \notag
&=\frac{\sqrt[]{a!a'!}}{m!}\sum_{\substack{{\bf  i} \in \{1,2,\ldots,\ell\}^{m} \\ \sharp_{i}{\bf i}=a_{i} \forall i \in \{1,2,\ldots,\ell\}}}\sum_{\substack{{\bf j} \in \{1,2,\ldots,\ell'\}^{m} \\ \sharp_{j}{\bf j}=a_{j}' \forall j \in \{1,2,\ldots,\ell'\}}}\prod_{n=1}^{m}\langle h_{i_{n}},h_{j_{n}}' \rangle_{{\mathbb H}}.
\end{align}
In the same way as \eqref{eq:4.8}, for $(k_{j}^{(i)}) \in {\mathbb Z}_{+}^{\ell} \otimes {\mathbb Z}_{+}^{\ell'}$ such that $\sum_{i=1}^{\ell}\sum_{j=1}^{\ell'} k_{j}^{(i)} = m$, we have
$$
\sharp\left\{
\begin{array}{l}({\bf i},{\bf j}) \in \{1,2,\ldots,\ell\}^{m} \times \{1,2,\ldots,\ell'\}^{m} \,;\, \\
\hspace{1.0cm}\sharp_{(i,j)} ({\bf i},{\bf j})=k_{j}^{(i)} \text{ for all } (i,j) \in \{1,2,\ldots,\ell\} \times \{1,2,\ldots,\ell'\}
\end{array}
\right\}
=\frac{m!}{\prod_{i=1}^{\ell}\prod_{j=1}^{\ell'}k_{j}^{(i)}!},
$$
where
$$
\sharp_{(i,j)}({\bf i},{\bf j}):=\sharp\left\{n \in \{1,2,\ldots,m\} \,;\, (i_{n},j_{n})=(i,j)\right\}
$$
for $({\bf i},{\bf j})=((i_{1},i_{2},\ldots,i_{m}),(j_{1},j_{2},\ldots,j_{m})) \in \{1,2,\ldots,\ell\}^{m} \times \{1,2,\ldots,\ell'\}^{m}$ and $(i,j) \in \{1,2,\ldots,\ell\} \times \{1,2,\ldots,\ell'\}$.
Thus, we obtain
\begin{align}
\label{eq:4.11}
&\sum_{\substack{{\bf i} \in \{1,2,\ldots,\ell\}^{m} \\ \sharp_{i}{\bf i}=a_{i} \forall i \in \{1,2,\ldots,\ell\}}}\sum_{\substack{{\bf j} \in \{1,2,\ldots,\ell'\}^{m} \\ \sharp_{j}{\bf j}=a_{j}' \forall j \in \{1,2,\ldots,\ell'\}}}\prod_{n=1}^{m}\langle h_{i_{n}},h_{j_{n}}' \rangle_{{\mathbb H}} \\ \notag
&=\sum_{\substack{{\bf i} \in \{1,2,\ldots,\ell\}^{m} \\ \sharp_{i}{\bf i}=a_{i} \forall i \in \{1,2,\ldots,\ell\}}}\sum_{\substack{{\bf j} \in \{1,2,\ldots,\ell'\}^{m} \\ \sharp_{j}{\bf j}=a_{j}' \forall j \in \{1,2,\ldots,\ell'\}}}\prod_{i=1}^{\ell} \prod_{j=1}^{\ell'} \langle h_{i},h_{j}' \rangle_{{\mathbb H}}^{\sharp_{(i,j)}({\bf i},{\bf j})} \\ \notag
&=\sum_{\substack{(k_{j}^{(i)}) \in {\mathbb Z}_{+}^{\ell} \otimes {\mathbb Z}_{+}^{\ell'} \\ k_{1}^{(i)}+k_{2}^{(i)}+\cdots+k_{\ell'}^{(i)}=a_{i} \forall i \in \{1,2,\ldots,\ell\} \\ k_{j}^{(1)}+k_{j}^{(2)}+\cdots+k_{j}^{(\ell)}=a_{j}' \forall j \in \{1,2,\ldots,\ell'\}}}
\sum_{\substack{({\bf i},{\bf j}) \in \{1,2,\ldots,\ell\}^{m} \times \{1,2,\ldots,\ell'\}^{m} \\ \sharp_{(i,j)}({\bf i},{\bf j})=k_{j}^{(i)} \forall (i,j) \in \{1,2,\ldots,\ell\} \times \{1,2,\ldots,\ell'\}}}\prod_{i=1}^{\ell}\prod_{j=1}^{\ell'} \langle h_{i},h_{j}' \rangle_{{\mathbb H}}^{k_{j}^{(i)}} \\ \notag
&=m!\sum_{\substack{(k_{j}^{(i)}) \in {\mathbb Z}_{+}^{\ell} \otimes {\mathbb Z}_{+}^{\ell'} \\ k_{1}^{(i)}+k_{2}^{(i)}+\cdots+k_{\ell'}^{(i)}=a_{i} \forall i \in \{1,2,\ldots,\ell\} \\ k_{j}^{(1)}+k_{j}^{(2)}+\cdots+k_{j}^{(\ell)}=a_{j}' \forall j \in \{1,2,\ldots,\ell'\}}}\prod_{i=1}^{\ell}\prod_{j=1}^{\ell'} \frac{\langle h_{i},h_{j}' \rangle_{{\mathbb H}}^{k_{j}^{(i)}}}{k_{j}^{(i)}!}.
\end{align}
By combining \eqref{eq:4.6}, \eqref{eq:4.10} and \eqref{eq:4.11}, we conclude the proof.
\end{proof}

To specifically calculate the value of the Hilbert--Schmidt inner product on the right-hand side of Lemma \ref{lem:4.3}, we need the following condition.

\begin{Definition}[match]
\label{def:4.4}
Let $N_{0},N_{1} \in {\mathbb N}$, and $a \in {\mathbb Z}_{+}^{N_{0}}$ and $a' \in {\mathbb Z}_{+}^{N_{0}N_{1}}$ be multi-indexes.
If for any $i \in \{1,2,\ldots,N_{0}\}$,
$$
\sum_{k=1}^{i}a_{k}=\sum_{k=1}^{iN_{1}}a'_{k},
$$
then we call $a'$ matches $a$.
\end{Definition}

As we can easily see that 
\begin{itemize}
\item if $a'$ matches $a$, then $|a|=|a'|$, and
\item $a'$ matches $a$ if and only if for any $i \in \{1,2,\ldots,N_{0}\}$,
$$
a_{i}=\sum_{k=(i-1)N_{1}+1}^{iN_{1}}a_{k}'.
$$
\end{itemize}

\begin{Lemma}
\label{lem:4.5}
Let $N_{0},N_{1} \in {\mathbb N}$, and $a \in {\mathbb Z}_{+}^{N_{0}}$ and $a' \in {\mathbb Z}_{+}^{N_{0}N_{1}}$ be multi-indexes such that $|a|=|a'|=:m$.
Then, it holds that
$$
\langle {\bf h}_{a}^{(N_{0})},{\bf h}_{a'}^{(N_{0}N_{1})} \rangle_{{\rm HS}}
=
\left\{
\begin{array}{ll}
\displaystyle{\sqrt[]{\frac{a!}{a'!}}\frac{1}{\sqrt[]{N_{1}}^{m}}} & \text{if } a'  \text{ matches } a \\
0 & \text{otherwise}.
\end{array}
\right.
$$
\end{Lemma}
\begin{proof}
We first obtain for any $i,j \in {\mathbb N}$,
\begin{align}
\label{eq:4.12}
\langle h_{i}^{(N_{0})},h_{j}^{(N_{0}N_{1})} \rangle_{{\mathbb H}}
&=\frac{\langle {\bf 1}_{(\frac{i-1}{N_{0}}T,\frac{i}{N_{0}}T]},{\bf 1}_{(\frac{j-1}{N_{0}N_{1}}T,\frac{j}{N_{0}N_{1}}T]} \rangle_{{\mathbb H}}}{T/(N_{0}\sqrt[]{N_{1}})}
=
\left\{
\begin{array}{ll}
\displaystyle
\frac{1}{\sqrt[]{N_{1}}} & \text{if } (i-1)N_{1}<j \leq iN_{1} \\
0 & \text{otherwise}.
\end{array}
\right.
\end{align}

We take $\sigma \in {\mathfrak S}_{m}$ which satisfies the condition
\begin{align}
\label{eq:4.13}
(\eta_{a}(n)-1)N_{1}<\eta_{a'}(\sigma(n)) \leq \eta_{a}(n)N_{1}, \quad n \in \{1,2,\ldots,m\}.
\end{align}
By the definition of $\eta_{a}$, the condition \eqref{eq:4.13} if and only if for any $i \in \{1,2,\ldots,N_{0}\}$,
$$
(i-1)N_{1}<\eta_{a'}(\sigma(n)) \leq iN_{1}, \quad n \in \left\{\sum_{k=1}^{i-1}a_{k}+1,\sum_{k=1}^{i-1}a_{k}+2,\ldots,\sum_{k=1}^{i}a_{k}\right\}.
$$
Thus for each $i \in \{1,2,\ldots,N_{0}\}$, the mapping $\sigma$ changes the order of the nicely aligned permutation as follows:
\begin{align*}
&(\eta_{a'}(\sum_{k=1}^{(i-1)N_{1}}a_{k}'+1),\eta_{a'}(\sum_{k=1}^{(i-1)N_{1}}a_{k}'+2),\ldots,\eta_{a'}(\sum_{k=1}^{iN_{1}}a_{k}')) \\
&=(\underbrace{(i-1)N_{1}+1,\cdots,(i-1)N_{1}+1}_{a_{(i-1)N_{1}+1}'},\underbrace{(i-1)N_{1}+2,\cdots,(i-1)N_{1}+2}_{a_{(i-1)N_{1}+2}'},\cdots,\underbrace{iN_{1},\cdots,iN_{1}}_{a_{iN_{1}}'})
\end{align*}
$$
\overset{\sigma}{\Longrightarrow} \quad (\underbrace{\eta_{a'}(\sigma(\sum_{k=1}^{i-1}a_{k}+1)),\eta_{a'}(\sigma(\sum_{k=1}^{i-1}a_{k}+2)),\cdots,\eta_{a'}(\sigma(\sum_{k=1}^{i}a_{k}))}_{a_{i}}).
$$
This implies that
$$
\sharp\left\{\sigma \in {\mathfrak S}_{m} \,;\, \sigma \text{ satisfies the condition \eqref{eq:4.13}}\right\}=
\left\{
\begin{array}{ll}
\displaystyle
\prod_{i=1}^{N_{0}}a_{i}!=a! & \text{if } a'  \text{ matches } a \\
0 & \text{otherwise}.
\end{array}
\right.
$$
Therefore, by \eqref{eq:4.4} and \eqref{eq:4.12}, we have
\begin{align*}
\langle {\bf h}_{a}^{(N_{0})},{\bf h}_{a'}^{(N_{0}N_{1})} \rangle_{{\rm HS}}
&=\frac{1}{\sqrt[]{a!a'!}}\sum_{\sigma \in {\mathfrak S}_{m}}\prod_{n=1}^{m}\langle h_{\eta_{a}(n)}^{(N_{0})},h_{\eta_{a'}(\sigma(n))}^{(N_{0}N_{1})} \rangle_{{\mathbb H}} \\
&=
\left\{
\begin{array}{ll}
\displaystyle{\sqrt[]{\frac{a!}{a'!}}\frac{1}{\sqrt[]{N_{1}}^{m}}} & \text{if } a' \text{ matches } a \\
0 & \text{otherwise}.
\end{array}
\right.
\end{align*}
\end{proof}

The following lemma gives estimates of some partial sums.

\begin{Lemma}
\label{lem:4.6}
Let $N_{0},N_{1} \in {\mathbb N}$, $n \in {\mathbb N}$, $r \in [0,1]$ and $a \in {\mathbb Z}_{+}^{N_{0}}$ be a multi-index.
Then, the following three statements hold:
\begin{enumerate}
\item
$$
\sum_{\substack{a' \in {\mathbb Z}_{+}^{N_{0}N_{1}} \\ a' \text{ matches } a}}\frac{a!}{a'!}\frac{1}{N_{1}^{|a|}}=1,
$$
\item
$$
\sum_{\substack{a' \in {\mathbb Z}_{+}^{N_{0}N_{1}} \\ a' \text{ matches } a \\ \exists \ell' \in {\mathbb N} \text{ s.t. } a'_{\ell'}>n \text{ and } a'_{i}=0 \forall i>\ell'}}\frac{a!}{a'!}\frac{1}{N_{1}^{|a|}} \leq \frac{|a|^{n}}{n!N_{1}^{n}},
$$
\item
$$
\sum_{\substack{a' \in {\mathbb Z}_{+}^{N_{0}N_{1}} \\ a' \text{ matches } a \\ \exists \ell' \in {\mathbb N} \text{ s.t. } a'_{\ell'}>n \text{ and } a'_{i}=0 \forall i>\ell'}}\frac{a!}{a'!}\frac{1}{N_{1}^{|a|}} \leq \left(\frac{|a|^{n}}{n!N_{1}^{n}}\right)^{r}.
$$
\end{enumerate}
\end{Lemma}
\begin{proof}
For $a' \in {\mathbb Z}_{+}^{N_{0}N_{1}}$ such that $a'$ matches $a$, we set $b^{(i)}:=(a_{(i-1)N_{1}+1}',a_{(i-1)N_{1}+2}',\ldots,a_{iN_{1}}')$, $i \in \{1,2,\ldots,N_{0}\}$.
Then for any $i \in \{1,2,\ldots,N_{0}\}$, $b^{(i)} \in {\mathbb Z}_{+}^{N_{1}}$ and $|b^{(i)}|=a_{i}$, and by using the multinomial theorem, we obtain
\begin{align*}
\sum_{\substack{a' \in {\mathbb Z}_{+}^{N_{0}N_{1}} \\ a' \text{ matches } a}}\frac{a!}{a'!}\frac{1}{N_{1}^{|a|}}
&=\sum_{\substack{b^{(1)} \in {\mathbb Z}_{+}^{N_{1}} \\ |b^{(1)}|=a_{1}}}\sum_{\substack{b^{(2)} \in {\mathbb Z}_{+}^{N_{1}} \\ |b^{(2)}|=a_{2}}} \cdots \sum_{\substack{b^{(N_{0})} \in {\mathbb Z}_{+}^{N_{1}} \\ |b^{(N_{0})}|=a_{N_{0}}}}\frac{a_{1}!a_{2}! \cdots a_{N_{0}}!}{b^{(1)}!b^{(2)}! \cdots b^{(N_{0})}!}\frac{1}{N_{1}^{a_{1}+a_{2}+\cdots+a_{N_{0}}}} \\
&=\left(\sum_{\substack{b^{(1)} \in {\mathbb Z}_{+}^{N_{1}} \\ |b^{(1)}|=a_{1}}}\frac{a_{1}!}{b^{(1)}!}\frac{1}{N_{1}^{a_{1}}}\right)\left(\sum_{\substack{b^{(2)} \in {\mathbb Z}_{+}^{N_{1}} \\ |b^{(2)}|=a_{2}}}\frac{a_{2}!}{b^{(2)}!}\frac{1}{N_{1}^{a_{2}}}\right) \cdots \left(\sum_{\substack{b^{(N_{0})} \in {\mathbb Z}_{+}^{N_{1}} \\ |b^{(N_{0})}|=a_{N_{0}}}}\frac{a_{N_{0}}!}{b^{(N_{0})}!}\frac{1}{N_{1}^{a_{N_{0}}}}\right)
=1.
\end{align*}
Thus, the statement (1) is proved.

For $a' \in {\mathbb Z}_{+}^{N_{0}N_{1}}$ such that $a'$ matches $a$ and there exists $\ell' \in {\mathbb N}$ such that $a_{\ell'}'>n$ and $a_{i}'=0$ for all $i>\ell'$, since $a'$ matches $a$, there exists $\ell \equiv \ell(\ell') \in {\mathbb N}$ such that $a_{\ell}>n$ and $a_{i}=0$ for all $i>\ell$.
Here, if $\ell=N_{0}$, the condition, $a_{i}=0$ for all $i>\ell$, is not taken into account.
Then in the same way as in the proof of (1), by using the multinomial theorem, we have
\begin{align*}
&\sum_{\substack{a' \in {\mathbb Z}_{+}^{N_{0}N_{1}} \\ a' \text{ matches } a \\ \exists \ell' \in {\mathbb N} \text{ s.t. } a'_{\ell'}>n \text{ and } a'_{i}=0 \forall i>\ell'
}} \frac{a!}{a'!}\frac{1}{N_{1}^{|a|}} \\
&=\sum_{\substack{b^{(1)} \in {\mathbb Z}_{+}^{N_{1}} \\ |b^{(1)}|=a_{1}}}\sum_{\substack{b^{(2)} \in {\mathbb Z}_{+}^{N_{1}} \\ |b^{(2)}|=a_{2}}} \cdots \sum_{\substack{b^{(\ell-1)} \in {\mathbb Z}_{+}^{N_{1}} \\ |b^{(\ell-1)}|=a_{\ell-1}}}\sum_{\ell''=1}^{N_{1}}\sum_{\substack{b^{(\ell)} \in {\mathbb Z}_{+}^{N_{1}} \\ |b^{(\ell)}|=a_{\ell} \\ b^{(\ell)}_{\ell''}>n \text{ and } b^{(\ell)}_{i}=0 \forall i>\ell''}}\frac{a_{1}!a_{2}! \cdots a_{\ell}!}{b^{(1)}!b^{(2)}! \cdots b^{(\ell)}!}\frac{1}{N_{1}^{a_{1}+a_{2}+\cdots+a_{\ell}}} \\
&=\left(\sum_{\substack{b^{(1)} \in {\mathbb Z}_{+}^{N_{1}} \\ |b^{(1)}|=a_{1}}}\frac{a_{1}!}{b^{(1)}!}\frac{1}{N_{1}^{a_{1}}}\right) \cdots \left(\sum_{\substack{b^{(\ell-1)} \in {\mathbb Z}_{+}^{N_{1}} \\ |b^{(\ell-1)}|=a_{\ell-1}}}\frac{a_{\ell-1}!}{b^{(\ell-1)}!}\frac{1}{N_{1}^{a_{\ell-1}}}\right)\sum_{\ell''=1}^{N_{1}}\sum_{\substack{b^{(\ell)} \in {\mathbb Z}_{+}^{N_{1}} \\ |b^{(\ell)}|=a_{\ell} \\ b^{(\ell)}_{\ell''}>n \text{ and } b^{(\ell)}_{i}=0 \forall i>\ell''}}\frac{a_{\ell}!}{b^{(\ell)}!}\frac{1}{N_{1}^{a_{\ell}}} \\
&=\sum_{\ell''=1}^{N_{1}}\sum_{\substack{b^{(\ell)} \in {\mathbb Z}_{+}^{N_{1}} \\ |b^{(\ell)}|=a_{\ell} \\ b^{(\ell)}_{\ell''}>n \text{ and } b^{(\ell)}_{i}=0 \forall i>\ell''}}\frac{a_{\ell}!}{b^{(\ell)}!}\frac{1}{N_{1}^{a_{\ell}}}
=\sum_{\ell''=1}^{N_{1}}\sum_{k=n+1}^{a_{\ell}}\sum_{\substack{b^{(\ell)} \in {\mathbb Z}_{+}^{N_{1}} \\ |b^{(\ell)}|=a_{\ell} \\ b^{(\ell)}_{\ell''}=k \text{ and } b^{(\ell)}_{i}=0 \forall i>\ell''}}\frac{a_{\ell}!}{b^{(\ell)}!}\frac{1}{N_{1}^{a_{\ell}}} \\
&=\sum_{\ell''=1}^{N_{1}}\sum_{k=n+1}^{a_{\ell}}\sum_{\substack{c \in {\mathbb Z}_{+}^{\ell''-1} \\ |c|=a_{\ell}-k}}\frac{a_{\ell}!}{c!k!}\frac{1}{N_{1}^{a_{\ell}}}
=\sum_{\ell''=1}^{N_{1}}\sum_{k=n+1}^{a_{\ell}}\sum_{\substack{c \in {\mathbb Z}_{+}^{\ell''-1} \\ |c|=a_{\ell}-k}}\frac{a_{\ell}!}{k!(a_{\ell}-k)!}\frac{(a_{\ell}-k)!}{c!}\frac{1}{N_{1}^{a_{\ell}}} \\
&=\sum_{\ell''=1}^{N_{1}}\sum_{k=n+1}^{a_{\ell}}\frac{a_{\ell}!}{k!(a_{\ell}-k)!}(\ell''-1)^{a_{\ell}-k}\frac{1}{N_{1}^{a_{\ell}}}.
\end{align*}
Here, for any $\ell'' \in \{1,2,\ldots,N_{1}\}$, since by Taylor's theorem,
$$
(\ell'')^{a_{\ell}}=\sum_{k=0}^{n}\frac{a_{\ell}!}{k!(a_{\ell}-k)!}(\ell''-1)^{a_{\ell}-k}+\frac{a_{\ell}!}{n!(a_{\ell}-(n+1))!}\int_{\ell''-1}^{\ell''}(\ell''-y)^{n}y^{a_{\ell}-(n+1)}{\rm d}y,
$$
we obtain
\begin{align*}
\sum_{k=n+1}^{a_{\ell}}\frac{a_{\ell}!}{k!(a_{\ell}-k)!}(\ell''-1)^{a_{\ell}-k}
&=\frac{a_{\ell}!}{n!(a_{\ell}-(n+1))!}\int_{\ell''-1}^{\ell''}(\ell''-y)^{n}y^{a_{\ell}-(n+1)}{\rm d}y \\
&\leq \frac{a_{\ell}!}{n!(a_{\ell}-(n+1))!}\int_{\ell''-1}^{\ell''}y^{a_{\ell}-(n+1)}{\rm d}y \\
&=\frac{a_{\ell}!}{n!(a_{\ell}-n)!}\left((\ell'')^{a_{\ell}-n}-(\ell''-1)^{a_{\ell}-n}\right).
\end{align*}
Hence, since $\sum_{\ell''=1}^{N_{1}}((\ell'')^{a_{\ell}-n}-(\ell''-1)^{a_{\ell}-n})=N_{1}^{a_{\ell}-n}$, we have
\begin{align*}
\sum_{\substack{a' \in {\mathbb Z}_{+}^{N_{0}N_{1}} \\ a' \text{ matches } a \\ \exists \ell' \in {\mathbb N} \text{ s.t. } a'_{\ell'}>n \text{ and } a'_{i}=0 \forall i>\ell'
}} \frac{a!}{a'!}\frac{1}{N_{1}^{|a|}}
&\leq \sum_{\ell''=1}^{N_{1}}\frac{a_{\ell}!}{n!(a_{\ell}-n)!}\left((\ell'')^{a_{\ell}-n}-(\ell''-1)^{a_{\ell}-n}\right)\frac{1}{N_{1}^{a_{\ell}}} \\
&\leq \frac{a_{\ell}^{n}}{n!N_{1}^{n}}.
\end{align*}
Thus, the statement (2) is proved.

From the statements (1) and (2), we obtain
\begin{align*}
\sum_{\substack{a' \in {\mathbb Z}_{+}^{N_{0}N_{1}} \\ a' \text{ matches } a \\ \exists \ell' \in {\mathbb N} \text{ s.t. } a'_{\ell'}>n \text{ and } a'_{i}=0 \forall i>\ell'}}\frac{a!}{a'!}\frac{1}{N_{1}^{|a|}}
&\leq \left(\sum_{\substack{a' \in {\mathbb Z}_{+}^{N_{0}N_{1}} \\ a' \text{ matches } a \\ \exists \ell' \in {\mathbb N} \text{ s.t. } a'_{\ell'}>n \text{ and } a'_{i}=0 \forall i>\ell'
}}\frac{a!}{a'!}\frac{1}{N_{1}^{|a|}}\right)^{r}\left(\sum_{\substack{a' \in {\mathbb Z}_{+}^{N_{0}N_{1}} \\ a' \text{ matches } a}}\frac{a!}{a'!}\frac{1}{N_{1}^{|a|}}\right)^{1-r} \\
&\leq \left(\frac{|a|^{n}}{n!N_{1}^{n}}\right)^{r}. 
\end{align*}
Therefore, all the statements are proved.
\end{proof}

Using the above lemmas, we prove Theorem \ref{theo:3.2}.

\subsection{Proof of Theorem \ref{theo:3.2}}
\label{sec:4.1}
\begin{proof}
Since ${\mathbb D}_{2,s}^{(N_{0})}$ is a subspace of ${\mathbb D}_{2,s}^{(N_{0}N_{1})}$, by the Fourier expansion \eqref{eq:3.1} of $F$, we obtain
\begin{align*}
\left\|{\rm Err}_{n}^{(N_{0}N_{1})}(F)\right\|_{2,s}^{2}
&=\left\|\sum_{a \in {\mathbb Z}_{+}^{N_{0}}}(1+|a|)^{s/2}{\mathbb E}\left[F{\bf H}_{a}({\bf h}^{(N_{0})})\right]{\rm Err}_{n}^{(N_{0}N_{1})}((1+|a|)^{-s/2}{\bf H}_{a}({\bf h}^{(N_{0})}))\right\|_{2,s}^{2} \\
&=\sum_{a \in {\mathbb Z}_{+}^{N_{0}}}(1+|a|)^{s}{\mathbb E}\left[F{\bf H}_{a}({\bf h}^{(N_{0})})\right]^{2}\left\|{\rm Err}_{n}^{(N_{0}N_{1})}((1+|a|)^{-s/2}{\bf H}_{a}({\bf h}^{(N_{0})}))\right\|_{2,s}^{2}
\end{align*}
Here, since the mapping ${\rm Err}_{n}^{(N_{0}N_{1})}$ is the projection from ${\mathbb D}^{(N_{0}N_{1})}_{2,s}$ onto the completion of the linear span of
$$
\left\{(1+|a'|)^{-s/2}{\bf H}_{a'}({\bf h}^{(N_{0}N_{1})}) \,;\, 
\begin{array}{l}
a'  \in {\mathbb Z}_{+}^{N_{0}N_{1}} \text{ such that there exists } \ell' \in {\mathbb N} \\
\text{such that } a_{\ell'}>n \text{ and } a_{i}'=0 \text{ for all } i>\ell'
\end{array}
\right\},
$$
by Lemma \ref{lem:4.3}, Lemma \ref{lem:4.5} and Lemma \ref{lem:4.6} (3), we have
\begin{align*}
\left\|{\rm Err}_{n}^{(N_{0}N_{1})}((1+|a|)^{-s/2}{\bf H}_{a}({\bf h}^{(N_{0})}))\right\|_{2,s}^{2}
&=\sum_{\substack{a' \in {\mathbb Z}_{+}^{N_{0}N_{1}} \\ \exists \ell' \in {\mathbb N} \text{ s.t. } a'_{\ell'}>n \text{ and } a'_{i}=0 \forall i>\ell'}}{\mathbb E}\left[{\bf H}_{a}({\bf h}^{(N_{0})}){\bf H}_{a'}({\bf h}^{(N_{0}N_{1})})\right] \\
&=\sum_{\substack{a' \in {\mathbb Z}_{+}^{N_{0}N_{1}} \\ \exists \ell' \in {\mathbb N} \text{ s.t. } a'_{\ell'}>n \text{ and } a'_{i}=0 \forall i>\ell'}}\langle {\bf h}_{a}^{(N_{0})},{\bf h}_{a'}^{(N_{0}N_1)} \rangle_{{\rm HS}} \\
&=\sum_{\substack{a' \in {\mathbb Z}_{+}^{N_{0}N_{1}} \\ a' \text{ matches } a \\ \exists \ell' \in {\mathbb N} \text{ s.t. } a'_{\ell'}>n \text{ and } a'_{i}=0 \forall i>\ell'}}\frac{a!}{a'!}\frac{1}{N_{1}^{|a|}} \\
&\leq \left(\frac{|a|^{n}}{n!N_{1}^{n}}\right)^{r}.
\end{align*}
Thus by \eqref{eq:2.1}, we obtain
\begin{align*}
\left\|{\rm Err}_{n}^{(N_{0}N_{1})}(F)\right\|_{2,s}^{2}
&\leq \sum_{a \in {\mathbb Z}_{+}^{N_{0}}}(1+|a|)^{s}{\mathbb E}\left[F{\bf H}_{a}({\bf h}^{(N_{0})})\right]^{2}\left(\frac{|a|^{n}}{n!N_{1}^{n}}\right)^{r} \\
&\leq  \sum_{a \in {\mathbb Z}_{+}^{N_{0}}}(1+|a|)^{s+rn}{\mathbb E}\left[F{\bf H}_{a}({\bf h}^{(N_{0})})\right]^{2}\frac{1}{(n!N_{1}^{n})^{r}} \\
&=\frac{\|F\|_{s+rn}^{2}}{(n!N_{1}^{n})^{r}}.
\end{align*}
Therefore, we conclude the statement.
\end{proof}

\section*{Summary}
\label{sec:5}
For $N_{0},N_{1} \in {\mathbb N}$, $n \in {\mathbb N}$, $s \in {\mathbb R}$ and $r \in [0,1]$, we have shown that it is $(n!N_{1}^{n})^{-r/2}$ that the ${\mathbb D}_{2,s}$-convergence rate of the $n$-th-order error of the $N_{0}N_{1}$-discrete-time Clark--Ocone formula provided by Akahori--Amaba--Okuma \cite{AJATOK} for Wiener functionals belongs to ${\mathbb D}_{2,s+rn}^{(N_{0})}$.
As applications, the higher-order error estimate might be able to be applied to estimates for the tracking error of the delta hedge or the delta--gamma hedge in mathematical finance, or to computational complexity arguments for control variate schemes in order to better truncate the sum of the discrete-time Clark--Ocone formula with considering the computational cost.

\section*{Acknowledgments}
\label{sec:6}
The authors would like to thank Jir\^o Akahori (Ritsumeikan Univ.) for his helpful comments.
The third author was supported by JSPS KAKENHI Grant Number 17J05514.

\end{document}